\documentclass{amsart}
\newtheorem{thm}{Theorem}[section]

\newtheorem{lem}{Lemma}[section]

\theoremstyle{definition}
\newtheorem{defn}{Definition}[section]
\theoremstyle{remark}
\newtheorem{rem}{Remark}[section]
\numberwithin{equation}{section}

\usepackage{hyperref}
\begin{document}

\title[Quasi-nearly subharmonic functions and quasiconformal mappings]{Quasi-nearly subharmonic functions and quasiconformal mappings}%
\author{Pekka Koskela and Vesna Manojlovi\'c}%
\thanks{The first author was supported by grants from the Academy of Finland and
the second author by MN Project 174024, Serbia.}
\address{Department of Mathematics and Statistics, P.O. Box 35, FIN-40014,
University of Jyv\"askyl\"a, Finland}%
\email{\tt pkoskela@maths.jyu.fi}%
\address{University of Belgrade, Faculty of Organizational Sciences,
Jove Ilica 154, Belgrade, Serbia}
\email{\tt vesnam@fon.rs}%
\thanks{}%

\maketitle

{\it $\ \ \ \ \ \ \ \ \ \ \ \ \ \ \ \ \ \ \ $ Dedicated to Professor Miroslav Pavlovic}

\begin{abstract}
We prove that the composition of a quasi-nearly subharmonic function and a
quasiregular mappings of bounded multiplicity is
quasi-nearly subharmonic.
Also, we prove that if $u\circ f$ is quasi-nearly subharmonic for all
quasi-nearly subharmonic $u$ and $f$
satisfies some additional conditions, then $f$ is quasiconformal.
Similar results are further established for the class of regularly
oscillating functions.
\end{abstract}


\section{Introduction and results}
Let $\Omega$ be a domain in the Euclidean space $\mathbf R^n$.
If $h$ is a function
harmonic in $\Omega$, then the function $|h|^p,$ which need not be
subharmonic in $\Omega$ for $0<p<1,$ behaves like a subharmonic function:
the inequality
\begin{equation}\label{h}
    |h(a)|^p\le \frac C{r^n}\int_{B(a,r)}|h|^p\, dm, \quad {B(a,r)}\subset
\Omega, \ 0<p<\infty,
\end{equation}
holds, whenever $0<p<\infty,$ $B(a,r)=\{x:|x-a|<r\}\subset \Omega$,
and $dm$ is the Lebesgue measure normalized so that $|B(0,1)|:=m(B(0,1))=1.$
The constant $C$ in \eqref{h} depends only on $n$ and $p$ when $p<1$,
and $C=1$ when $p\ge1.$ This fact is essentially due to Hardy and
Littlewood (see \cite[Theorem 5]{HL}), although they never formulated it.
The proof was first given by Fefferman and Stein \cite{FS},
and independently by Kuran \cite{KU}. It follows from Fefferman and Stein's
proof that \eqref{h} remains true if $|h|$ is replaced by a
nonnegative subharmonic function.
Hence:
 \begin{thm}\label{fs} If $u\ge0$ is a function subharmonic in a domain
$\Omega\subset \mathbf R^n,$ then
 \begin{equation}\label{uu}
    u(a)^p\le \frac C{r^n}\int_{B(a,r)}u^p\, dm, \quad B(a,r)\subset \Omega,
\ 0<p<\infty,
\end{equation}
where $C$ depends only on $p$ and $n$, when $p<1$, and $C=1$ when $p\ge 1.$
 \end{thm}

Let $u\ge 0$ be a locally bounded, measurable function on $\Omega$. We say
(see \cite{Pa}, \cite{PR}) that $u$ is $C$-{\it quasi-nearly subharmonic}
(abbreviated $C$-qns)
if the following condition is satisfied:
 \begin{equation}\label{uu2}
    u(a)\le \frac C{r^n}\int_{B(a,r)}u\, dm, \quad
\text{whenever $B(a,r)\subset \Omega$.}
\end{equation}

One can view (\ref{uu2}) as a weak mean value property. Besides of
nonnegative subharmonic functions it also holds for nonnegative subsolutions
to a  large family of second order elliptic equations, see \cite{HKM}.
In fact, (\ref{uu2}) is typically proven as a step towards Harnack inequalities
for second order elliptic equations, using the Moser iteration scheme.
Notice that $u$ is quasi-nearly subharmonic if and only if $u$ is everywhere
dominated by its centered minimal function \cite{C}.

Our first result is an invariance property.

\begin{thm}\label{km} If $u\ge0$ is a $C$-qns function
defined on a domain $\Omega'\subset\mathbf R^n$, $n\ge 2,$
and $f$ is a $K$-quasiregular mapping, with bounded multiplicity $N$,
from a domain $\Omega$ onto $\Omega',$ then the function $u\circ f$ is
$C_1$-qns in $\Omega,$ where $C_1$ only depends on $K$, $C$, $N$, and $n.$
  \end{thm}
\begin{rem}
The hypothesis of bounded multiplicity of $f$ is necessary as the 
following  example shows. Let $f(z)=e^z$, $\Omega=\mathbb C,$ 
$\Omega'=\mathbb C\setminus\{0\},$ 
$$ E=\bigcup_{j\ge2}[\exp(2^j),\exp(2^j+1)],$$
and $u(w)=\chi_E(|w|).$
Then it is easy to check that $u$ is quasi-nearly subharmonic in $\Omega'$ 
but $u\circ  f$ is not quasi-nearly subharmonic in $\Omega.$
\end{rem}
Above quasiregularity requires that $f$ is continuous, the component functions
of $f$ belong locally to the Sobolev class $W^{1,n}$ and that there is
a constant $K\ge 1$ so that
$$|Df(x)|^n\le K J(x,f)$$
holds almost everywhere in $\Omega.$ Injective quasiregular mappings are
called quasiconformal.
It was previously only known that the invariance holds under conformal mappings
in the planar case \cite{VM} and under bi-Lipschitz mappings \cite{DR}
in all dimensions. Let us consider the above morphism property in more detail.


\begin{defn}
Let $\Omega$ and $\Omega'$ be subdomains of $\mathbf R^n.$
A mapping
$f:\Omega\mapsto \Omega'$ is a {\it qns-morphism} if there is a
constant $C<\infty$
such that for every qns $u$ defined in $\Omega'$ we have
\[\|u\circ f\|_{\rm qns}\le C\|u\|_{\rm qns},\]
where
\[\|u\|_{\rm qns}=\inf\Big\{C\ge 0: u(a)\le
\frac C{r^n}\int_{B(a,r)}u\,dm\
\text{for all $a\in \Omega',\ 0<r\le d(x,\partial \Omega')$}\Big\}.\]
If the above holds with a constant $C,$ we call $f$ a
$C$-qns-morphism. Finally, $f$ is a {\it strong qns-morphism}
if there is a constant $C$ so that $f$ restricted
to any domain $G\subset \Omega $, $f:G\mapsto G'$, is a $C$-qns-morphism.
\end{defn}

\begin{thm}\label{new-pekka}
Let $\Omega,\,\Omega'\subset \mathbb R^n$, $n\ge2,$ be domains.
Then a homeomorphism $f:\Omega\mapsto \Omega'$
is a strong qns-morphism if and only if $f$ is quasiconformal.
\end{thm}

If we assume sufficient a priori regularity for $f,$  a version of
Theorem \ref{new-pekka} holds also for qns-morphisms.
The reader may wish to compare this with related quasiconformal invariance
properties for other function classes \cite{A}, \cite{R}, \cite{S1},
\cite{S2}, \cite{U}.


\begin{thm}\label{km-conv}
Let $n\ge 2$ and let $f:\Omega\mapsto\Omega'$ be a qns-morphism
that belongs to $W^{1,n}_{\rm loc}$.
If, additionally, $J(x,f)\ge 0$ almost everywhere, then $f$ is quasiregular.
\end{thm}

Each quasiregular mapping $f$ is either constant or both open and discrete;
in the latter case the multiplicity of $f$ is locally finite.  
The condition $J(x,f)\ge 0$ in Theorem \ref{km-conv} cannot
be dropped, as the mapping $f(x,y)=(x,|y|)$ is a planar (strong) qns-morphism.
The Sobolev regularity assumption can be slightly
relaxed: if $f$ above is a $C$-qns-morphism, then local $p$-integrability
of the distributional derivatives suffices for $p=p(n,C)<n;$ this
can be inferred from the proof of Theorem \ref{km-conv} using \cite{I}.
In the planar, injective setting, even $W^{1,1}_{\rm loc}$ suffices.




Let us close this introduction by commenting on the invariance of
a related function class, introduced in \cite{Pa}.

\begin{defn}
A function $u:\Omega'\mapsto \mathbb R^k$ is said to be {\it
regularly oscillating}
if
\begin{equation}\label{osc}
    {\rm Lip}\,u(x)\le Cr^{-1}\sup_{y\in B(x,r)\subset\Omega'}|u(y)-u(x)|,\quad x\in \Omega',\ B(x,r)\subset \Omega',
\end{equation}
where $C\ge 0$ is a constant independent of $x$ and $r$. Here
\[{\rm Lip}\,u(x)=\limsup_{y\to x}\frac{|u(y)-u(x)|}{|y-x|}.\]
Note that ${\rm Lip}\,u(x)=|{\rm grad}\, u(x)|$ if $u$ is differentiable at
$x.$
The smallest $C$ satisfying \eqref{osc} will be denoted by $\|u\|_{\rm ro}.$
\end{defn}

We have the following invariance.

\begin{thm}\label{ro}
Let $f:\Omega\mapsto\Omega'$ be quasiregular, regularly oscillating and of
bounded multiplicity in $\Omega$. If $u$ is regularly oscillating in
$\Omega'$, then $u\circ f$ is regularly oscillating in
$\Omega$ with $\|u\circ f\|_{\rm ro}\le C'\|u\|_{\rm ro}$,
where $C'$ depends only on the multiplicity of $f$.
\end{thm}

The assumption that $f$ be regularly oscillating is necessary, as seen
by noticing that the coordinate projections are regularly oscillating; not
all quasiregular mappings are regularly oscillating.
In the case of an analytic function, this can naturally be dropped.
Similarly to Theorem \ref{km-conv}, quasiregularity is necessary if
we assume that $J(x,f)\ge 0$ almost everywhere, but no a priori Sobolev
regularity is needed because regularly oscillating functions and mappings
are locally
Lipschitz continuous. The invariance property of Theorem \ref{ro} was
established in \cite{VM} when
$f$ is conformal (and $n=2$).

\begin{rem}
The assumption of bounded multiplicity of $f$ in Theorem \ref{ro}
is necessary as in the case of Theorem \ref{km}. To see this simply let
$f(z)=e^z,$ $\Omega=\mathbb C$,  $\Omega'=\mathbb C\setminus\{0\},$
$$ E=\bigcup_{j\ge 2}[\exp(2^j),\exp(2^j+1)],$$
and $v(w)=\int_0^{|w|}\chi_E(t)\,dt$. Then $v$ is regularly oscillating 
but $v\circ f$ is not.
\end{rem}

\section{Proof of the theorem \ref{km}}
For the proof we need some lemmas. The first says that if $u^p$ is qns
for some $p$, then so is $u.$
\begin{lem}\cite{Pa}\label{pp}
If $u$ is $C$-qns, and $p>0,$ then $u^p$ is $C_1(C,n)$-qns.
\end{lem}

We also need the following lemma that can be distilled from the arguments
in \cite{HKms}. For the sake of completeness, we give a short proof below.

\begin{lem}\label{vesna}
Let $f:\Omega\mapsto \Omega'$ be $K$-quasiregular and of bounded
multiplicity $N.$ Let $x\in \Omega$ and $0<r\le \frac12 d(x,\partial\Omega
r)).$ Then
\begin{align*}
d\big(f(x),\partial f(B(x,
r))\big)\ge
\delta \sup_{y\in\overline B(x,
r)}|f(y)-f(x)|,
\end{align*}
where $\delta=\delta(n,K,N).$
\end{lem}
\begin{proof}
Let $x\in \Omega$ and let $0<r\le \frac 12 d(x,\partial\Omega).$
Now $f(x)$ is an interior point of $f(B(x,r))$ because $f$ is open.
Moreover
\[\sup_{y\in\overline B(x,r)}|f(y)-f(x)|=|f(z)-f(x)|\]
for some $z\in \partial B(x,r)$, and
\[0<d(f(x),\partial f(B(x,r))=|f(\omega)-f(x)|\]
for some $\omega\in \partial B(x,r)$.

Let $E=[f(x),f(\omega)]$ be the segment between $f(w)$ and $f(\omega),$
and $F$ be a segment that joins $f(z)$ to
$\partial \Omega'$ (or to infinity) outside the ball
\[B(f(x),|f(z)-f(\omega)|).\]
We may assume  that
\[|f(z)-f(x)|\ge 2|f(\omega)-f(x)|.\]
Let
\[ u(y)=
\begin{cases}
1, & y\in \overline B(f(x),|f(\omega)-f(x)|),\\
0, &  y\in B^c(f(x),|f(z)-f(x)|),\\
\log\frac{|f(z)-f(x)|}{|y-f(x)|}\Big/\log\frac{|f(z)-f(x)|}{|f(\omega)-f(x)|}\,& \text{elsewhere}.
\end{cases}
\]
Then, by a change of variables, see page 21 in \cite{Ri},
\begin{align*}
\int_\Omega |\nabla (u\circ f)|^n\, dm&\le K\int_{\Omega'}|(\nabla u)(f(y))|^n J_f(y)\,dm(y)\\&
\le KN\int_{\Omega'}|\nabla u|^n\,dm\\[0.5ex]&
\le \frac{KNC_n}{\log^{n-1}\dfrac{|f(z)-f(x)|}{|f(\omega)-f(x)|}}.
\end{align*}

On the other hand, since $f$ is open, the set $f^{-1}([f(x),f(\omega)])$
contains a continuum joining $x$ to $\partial B(x,r)$, and $f^{-1}(F)$ a
continuum joining $\partial B(x,r)$ to $\partial B(x,\frac32r)$.
By usual capacity estimates, see e.g. \cite{HKM}
\[\int_{\Omega}|\nabla(u\circ f)|^n\,dm\ge \delta_0(n,K)>0.\]
The claim follows.
\end{proof}


As an immediate consequence of Lemma \ref{vesna} we have:
\begin{lem}\label{prec}
Let $B=B(0,1)$ and let $f:2B\mapsto\Omega'$ be a $K$-qr mapping with
bounded multiplicity $N$ and such that $f(0)=0.$
Then there exist $\rho\in (0,1)$ and $R>0$ such that
$B(0,R)\supset f(B)\supset B(0,\rho),$ where $R/\rho\le 1/\delta,$ and
$\delta$ depends only on $K,n,N.$
\end{lem}

Finally we need the following fundamental fact:
\begin{lem}\label{HK}\cite[p. 258]{HK}
Under the hypotheses of Lemma \ref{prec}, there exists $p>1$ such that
\[\ \Big(\int_BJ(y,f)^p\,dm\Big)^{1/p}\le C\int_B J(y,f)\,dm,\]
where $p$ depends only on $K,N$ and $n.$
\end{lem}

\subsection*{Proof of Theorem \ref{km}}
As is easily seen, the proof reduces to the case
$\Omega=B(0,2),$ $f(0)=0.$ Let $B=B(0,1)$ and write $v=u\circ f.$
By using translations and rotations, we see that it is enough to prove
that
\[ v(0)\le C\int_B v(y)\,dm(y),\]
where $C=C(K,n,\|u\|_{\rm qns}).$
By Lemma \ref{pp}, it suffices to find $q=q(K,N,n)\ge 1$ so that
\begin{equation}\label{}
    v(0)\le C\Big(\int_B v(y)^q\,dm(y)\Big)^{1/q}.
\end{equation}
To prove this, we start from H\"older's inequality:
\[\int_Bv(y)J(y,f)\,dm(y)\le \Big(\int_Bv(y)^q\,dm(y)\Big)^{1/q}\Big(\int_BJ(y,f)^p\,dm(y)\Big)^{1/p},\]
where $p=q/(q-1).$
By a change of variables, see page 21 in \cite{Ri}, we have
\[\begin{aligned}\int_{B}v(y)J(y,f)\,dm(y)&=\int_{f(B)}u(y)\,N(y,f,B)\,dm(y)\\&\ge \int_{f(B)}u(y)\,dm(y)\\&
\ge \int_{B(0,\rho)}u(y)\,dm(y)\\& \ge c\rho^n u(0)=c\rho^n v(0).
\end{aligned}
\]
Here we have used Lemma \ref{prec} and the hypothesis that $u$ is qns. On the other hand, by Lemmas \ref{HK} and \ref{prec}, we have
\[\begin{aligned}\Big(\int_{B}J(y,f)^p\,dm(y)\Big)^{1/p}\,dy&\le
C\int_B J(y,f)\,dm(y)\\&
=C\int_{f(B)}N(y,f,B)\,dm(y)\\&\le CN|f(B)|\\&\le  CN|B(0,R)|=CNk_nR^n.
\end{aligned}
\]
Combining these inequalities, we obtain
\[\begin{aligned}
c\rho^n v(0)\le CNk_nR^n\Big(\int_B v(y)^q\,dm(y)\Big)^{1/q}.
\end{aligned}
\]
Hence
\[v(0)\le\frac{CNk_nR^n}{c\rho^n}\Big(\int_B v(y)^q\,dm(y)\Big)^{1/q}. \]
Now the desired result follows from the inequality $R/\rho\le1/\delta,$ where $\delta$ depends only on $K,n,N.$

\section{Proof of Theorem \ref{km-conv}}

Even though our definition of quasiregular mappings requires them to be
continuous, this condition is superfluous and it suffices to show
that there exists $K\ge 1$ so that
$$|Df(x)|^n\le KJ(x,f)$$
holds almost everywhere, see e.g. page 177 in \cite{Ri}.
Next, every mapping $f$ with Sobolev regularity $W^{1,1}_{\rm loc}$ is
approximatively differentiable almost everywhere. That is, for almost every
$x_0$ and every $\epsilon>0,$ the set
$$A_{\epsilon}:=\{x:\ \frac{|f(x)-f(x_0)-Df(x_0)(x-x_0)|}{|x-x_0|}<\epsilon\}$$
has density one at $x_0,$ see e.g. page 140 in \cite{Z}.
Because of our a priori Sobolev regularity, it thus suffices to show the
above distortion inequality at every such point $x_0.$

For simplicity, we only give the proof in the planar case, assuming
differentiability instead of approximate differenentiability. The higher
dimensional setting and the switch to approximate differentiability only
require technical modifications that should be obvious to the reader after
examining the argument below. Thus suppose that $f$ is differentiable at
$x_0.$

{\em Case} (a). Suppose $f$ is differentiable at $x_0$ with
$J_f(x_0)\ne 0.$ In some coordinate systems we have
\[
Df(x_0)=\begin{bmatrix}
a & 0\\
0& b
\end{bmatrix}
\]
Assume $0<a<b$. We want to show that $b/a$ is bounded. Consider the function
\[ u(\omega)=\chi_{\{\omega=x'+iy': 0\le y'\le x'\}}(\omega-f(x_0)).\]
Then $\|u\|_{\rm qns}=8.$
Now
\begin{align*}
T^{-1}\big(\{\omega=x'+iy': 0\le y'\le x'\}\big)&
=\{z=x+iy: 0\le by\le ax\}\\&
=\{z=x+iy: 0\le y\le (a/b)x\}
\end{align*}
for the linear transformation $T$ associated to $Df(x_0).$
Now $u\circ f(x_0)=1.$ If $r>0$ is such that $B(x_0,r)\subset \Omega,$ we
conclude from the morphism property of $f$ that
\begin{align*}
1&\le \frac{C}{r^2}\int_{B(0,r)}u\circ f\,dm\\&=
\frac{C}{r^2}\frac 12r^2\arctan\frac ab+o(r)\\&
\to  \frac C2\frac ab,
\end{align*}
when $r\to 0,$ where $C>0$ comes from the morphism property.
Hence $b/a\le C/2$.
\par\medskip
{\it Case} (b).  Now suppose that $f$ is differentiable at $x_0$
with $J_f(x_0)=0$.
We want to prove that $Df(x_0)=0.$
We argue by contradiction: suppose $Df(x_0)\ne 0.$ We may assume
\[Df(x_0)=\begin{bmatrix}
0 & 0\\
0& 1
\end{bmatrix}.
\]
Define
\[ u(\omega)=\chi_{\{\omega=x'+iy': 0\le |y'|\le x'\}}(\omega-f(x_0)).\]
Then $\|u\|_{\rm qns}=4,$
and
\[ f^{-1}(t+it+f(x_0))=s_1(t)+is_2(t)+x_0.\]
where $\lim_{t\to 0}\dfrac{s_2(t)}{s_1(t)}=0.$
But then there is no $C>0$ such that
\[(u\circ f)(x_0)\le \frac C{r^2}\int_{B(x_0,r)}u\circ f\, dm\]
for all small $r>0,$
which contradicts the morphism property of $f.$

\section{Proof of Theorem \ref{new-pekka}}
It is an immediate consequence of Theorem \ref{km} that quasiconformality
of the homeomorphism  $f$ is a sufficient condition for $f$
to be a strong qns-morphism.

 In the other direction, it suffices to prove that $f^{-1}:\Omega'\mapsto \Omega$ is quasiconformal.
 Thus it suffices to verify the existence of $H<\infty$ such that
 \begin{equation}\label{razmak}
    \limsup_{r\to 0}\frac{{\rm diam}\,(f^{-1}(\overline B(y,r)))^n}{|f^{-1}(\overline B(y,r))|}\le H
 \end{equation}
 for all $y\in\Omega'$, see page 64 in \cite{HK}.

To simplify our notation,
we write $x'=f(x)$ for $x\in \Omega,$ in what follows.

 Fix $y'\in\Omega'$ and let $r>0.$
Towards proving \eqref{razmak}, we may assume that $r$ is so small that
 \[B(y,2\,{\rm diam}\,(f^{-1}(\overline B(y',2r))))\subset\Omega.\]
 Fix $y_0'\in \partial B(y',2r)$ and pick $y_1'\in\partial B(y',r)$ so that
 $$|y_0'-y_1'|=\max_{\omega'\in\partial B(y',r)}|\omega-y_0'|.$$
 Set $G'=\Omega'\setminus\{y_0'\}$ and $G=\Omega\setminus \{y_0\}.$
 Now $B(y_1,|y_1-y_0|/2)\subset G$ and
 \begin{equation}\label{blesav}
    {\rm diam\,}(f^{-1}(\overline B(y',r))\le 2|y_1-y_0|.
 \end{equation}
 Define $u(\omega')=\chi_{\overline B(y',r)}(\omega')$ for $\omega'\in G'.$
 Then $u$ is qns in $G'$ with $\|u\|_{\rm qns}\le 3^n.$
 Since $f$ is is $C$-qns-morphism in $G$, we conclude that
 \[\begin{aligned}&u\circ f(y_1)\le C3^n\frac1{| B(y_1,|y_1-y_0|/2)|}\int_{B(y_1,|y_1-y_0|/2)|}u\circ f\,dm\\[1ex]&
 =C3^n\frac{|f^{-1}(\overline B(y',r))\cap  B(y_1,|y_1-y_0|/2)|}{|B(y_1,|y_1-y_0|/2)|}\end{aligned}.\]
 Recalling \eqref{blesav} and that $u\circ f(y_1)=u(y_1')=1,$ we arrive at
 \[{\rm diam\,}(f^{-1}(\overline B(y',r))^n\le CC_n|(f^{-1}(\overline B(y',r))|\]
 as desired.
\section{Proof of Theorem \ref{ro}}

\subsection{Proof of Theorem \ref{ro}} Let $x\in \Omega$ and
$0<r<\frac 12d(x,\partial \Omega).$
Since the mapping $f$ is regularly oscillating and quasiregular we have,
by Lemma \ref{vesna},
\begin{align*}
{\rm Lip}\, f(x)&\le Cr^{-1}\sup_{y\in\overline B(x,r)}|f(y)-f(x)|\\&
\le \frac C\delta r^{-1} d(f(x),\partial f(B(x,r))).
\end{align*}

Recall that non-constant quasiregular mappings are open.
Since $u$ is regularly oscillating and $d(f(x),\partial f(B(x,r)))>0,$ we
have that
\begin{align*}
{\rm Lip}\,u( f(x))&\le \hat Cd(f(x),\partial f(B(x,r)))^{-1}\sup_{z\in B(f(x),d(f(x),\partial f(B(x,r)))}
|u(f(x))-z|\\&\le
\hat Cd(f(x),\partial f(B(x,r)))^{-1}\sup_{y\in B(x,r)}|u\circ f(y)-u\circ f(x)|.
\end{align*}
Now we have
\begin{align*}
{\rm Lip}\,(u\circ f)(x)&\le {\rm Lip}\,(u( f(x))\,{\rm Lip}\,f(x)\\&
\le \hat Cd(f(x),\partial f(B(x,r)))^{-1}\sup_{y\in B(x,r)}|u\circ f(y)-u\circ f(x)|\\&\quad\times \frac C\delta r^{-1} d(f(x),\partial f(B(x,r)))\\&
=C'r^{-1}\sup_{y\in B(x,r)}|u\circ f(y)-u\circ f(x)|.
\end{align*}
This completes the proof of the theorem.


\bibliographystyle{amsplain}

\end{document}